\documentclass{amsart}

\usepackage{times}

\usepackage[T1]{fontenc}
\usepackage{amsmath}
\usepackage{amssymb}
\usepackage{amsthm}
\usepackage{graphicx}
\usepackage{color}
\usepackage{calrsfs}
\usepackage[english]{babel}
\usepackage[height=1.5em, width=2em, nohug, midshaft]{diagrams}

\theoremstyle{plain}
\newtheorem{theorem}{Theorem}[section]
\newtheorem{lemma}[theorem]{Lemma}
\newtheorem{proposition}[theorem]{Proposition}
\newtheorem{corollary}[theorem]{Corollary}
\theoremstyle{definition}
\newtheorem{definition}[theorem]{Definition}
\newtheorem{remark}[theorem]{Remark}
\newtheorem{example}[theorem]{Example}

\newcommand{\Aff}{\mathbf{A}} % affine space \Aff^n
\newcommand{\ZZ}{\mathbf{Z}} % integers
\newcommand{\CC}{\mathbf{C}} % the complex number field
\newcommand{\suchthat}{\;\vline\;} % 'such that' set notation bar
\newcommand{\iso}{\cong}      % isomorphism
\newcommand{\PP}{\mathbf{P}}  % projective space
\newcommand{\sheaf}[1]{\mathcal{#1}} % use script font for sheaves
\newcommand{\stack}[1]{\mathcal{#1}} % use script font for stacks
\newcommand{\OO}{\sheaf{O}}   % structure sheaf of X is written \OO_X
\newcommand{\dual}[1]{{#1}^\vee}
\newcommand{\res}[2]{\left.#1\right|_{#2}} % restriction bar
\DeclareMathOperator{\hcf}{hcf} % highest common factor (gcd in US)
\DeclareMathOperator{\SL}{SL} % Special linear group
\DeclareMathOperator{\Proj}{Proj} % Projective spectrum
\DeclareMathOperator{\Spec}{Spec} % Affine spectrum
\newcommand{\PProj}{\mathcal{P}\mathit{roj}} % Stacky projective spectrum
\newcommand{\Gm}{{\mathbf{G}_m}} % Multiplicative group Spec k[t, 1/t]
\DeclareMathOperator{\Stab}{Stab} % Stabilizer subgroup
\DeclareMathOperator{\Aut}{Aut} % Automorphism group
\newcommand{\runity}{\boldsymbol{\mu}} % multiplicative group of roots of 1
\newcommand{\rig}[1]{{#1}^{\mathrm{rig}}} % rigidification of a stack
\newcommand{\can}[1]{{#1}^{\mathrm{can}}} % canonical stack of a scheme
\newcommand{\sstab}[1]{\ensuremath{{#1}^{\mathrm{ss}}}} % semistable locus
\newcommand{\setmin}{\smallsetminus}

\numberwithin{equation}{section}

\author{Martin G. Gulbrandsen}
\email{gulbr@kth.se}
\begin{document}
\address{Royal Institute of Technology\\Stockholm, Sweden}
\thanks{The author gratefully acknowledges support from Institut
Mittag-Leffler (Djursholm, Sweden), where part of this work was done.}

\title[Stack structures on GIT quotients]{Stack structures on GIT quotients parametrizing hypersurfaces}

% 14L24: Geometric invariant theory
% 14L30: Group actions on varieties or schemes (quotients)
% 14A20: Generalizations (algebraic spaces, stacks)

\begin{abstract}
We suggest to endow Mumford's GIT quotient scheme with a stack
structure, by replacing $\Proj(-)$ of the invariant ring with
its stack theoretic analogue. We analyse the stacks resulting in
this way from classically studied invariant rings, and in
particular for binary forms of low degree. Our viewpoint is that
the stack structure carries interesting geometric information
that is intrinsically present in the invariant ring, but lost
when passing to its $\Proj(-)$.
\end{abstract}
\maketitle

\section{Introduction}

Let $G$ be a reductive group acting on $\PP^N$ via a linear
representation, and let $Y\subseteq\PP^N$ be a $G$-invariant
subscheme with homogeneous coordinate ring $S$. Thus we consider
a linearized action of $G$ on $Y=\Proj(S)$. Let $R=S^G$ be the
invariant ring. According to Mumford's geometric invariant
theory, the semistable locus $\sstab{Y}$ admits a good quotient,
which is the projective scheme
\begin{equation*}
X = \Proj(R).
\end{equation*}
In classical invariant theory, a central question was to find
explicit presentations for the invariant ring $R$ in specific
examples. Such presentations give explicit equations for the GIT
quotient scheme $X$.

Let $0\in\Spec(R)$ be the vertex, defined by the ideal $R_+$
generated by elements of strictly positive degree.
Then $\Gm$ acts on the complement
$\Spec(R)\setmin \{0\}$, and the quotient scheme is
$\Proj(R)$.
The $\Gm$-action is
free if $R$ is generated in degree $1$, but not in general.  The
invariant rings we will consider are not generated in degree
$1$, and thus it is natural to consider also the stack quotient
$\stack{X}$ of the same action of $\Gm$ on $\Spec(R)\setmin
\{0\}$. This stack will be called the \emph{stacky GIT
quotient}. 

Thus the stacky GIT quotient $\stack{X}$ is a Deligne-Mumford
stack with the usual projective GIT quotient scheme $X$ as underlying
coarse space.
In the language of
Alper's stack theoretic treatment of GIT \cite{alper2008}, the scheme
$X$ is a ``good moduli space'' for the quotient stack $[\sstab{Y}/G]$,
and as the natural map $\sstab{Y}\to \stack{X}$ is $G$-invariant (in
the $2$-categorical sense), the quotient map from $[\sstab{Y}/G]$ to
its good moduli space factors through the stacky GIT quotient:
\begin{equation*}
[\sstab{Y}/G] \to \stack{X} \to X
\end{equation*}
Thus the stacky GIT quotient sits somewhere between the full
stack quotient and the GIT quotient scheme.
It has richer structure
than the latter, but is simpler than the full stack quotient, which
is not Deligne-Mumford in general.
On the other hand it is unclear
exactly what the stacky enrichment of the usual GIT quotient
captures, and we do not know of any sense in which it is a
\emph{quotient}, for instance in terms of a universal property.

Our aim is to analyse the relationship between the stacky GIT
quotient and the GIT quotient scheme in the examples studied in
classical invariant theory, where explicit presentations for the
invariant ring is known. Thus we are concerned with the action
of $G=\SL(n+1)$ by substitution on the projective space $Y$ of
degree $d$ hypersurfaces in $\PP^n$ for $n$ and $d$ small. More
precisely, we consider the actions of $\SL(2)$ on binary
quartics, quintics and sextics; of $\SL(3)$ on cubic curves; and
of $\SL(4)$ on cubic surfaces. 

The invariant ring of binary quartics and cubic curves are just
weighted polynomial rings in two variables. The invariant rings of
binary quintics, binary sextics, and cubic surfaces are more interesting,
and admits a special presentation (see
Equation \ref{eq:typicalR}) involving a certain polynomial $F$. We find that for
a ring $R$ of this form,
the corresponding stack $\stack{X}$ can essentially (precisely, up
to an essentially trivial $\runity_r$-gerbe) be reconstructed from
its coarse space $X=\Proj(R)$ together with the divisor $Z(F)$
on it. On the other hand, this divisor cannot be described in terms
of the intrinsic geometry of the GIT quotient \emph{scheme}. Thus, for invariant rings of the
type \eqref{eq:typicalR}, the divisor $Z(F)$ is
the essential piece of information that is intrinsically present in
the ring, and remembered by the stacky GIT quotient,
but ``forgotten'' by the GIT quotient scheme.

Moreover, in the case of $\SL(2)$ acting on binary forms of degree up
to $6$, we consider the classification of binary forms according to
symmetry, i.e.~their stabilizer groups in $\SL(2)$: Binary forms
with prescribed symmetry group correspond to locally closed loci
in the GIT quotient scheme $X$. We observe that the loci of binary
forms with extra symmetries (i.e.~larger symmetry group than the
generic one) occur as
\begin{enumerate}
\item singular points of $X$,
\item the divisor $Z(F)$,
\item singular points of $Z(F)$, or
\item singular points of the singular locus of $Z(F)$.
\end{enumerate}
As the stacky GIT quotient $\stack{X}$ remembers the divisor $Z(F)$, this enables
us to describe the loci of binary forms with extra symmetries in terms
of the intrinsic geometry of $\stack{X}$. We remark that the same statement
trivially holds for the stack quotient $[\sstab{Y}/G]$, but is not obvious
for the stacky GIT quotient, since the automorphism groups of its points
do not coincide with the stabilizer groups for the $\SL(2)$ action.

The approach in this text is entirely dependant on the invariant ring having the special
presentation \eqref{eq:typicalR}. This structure is
very special, although it is typical for the invariant rings determined
explicitly by the invariant theorists of the 19th century. Already for binary
forms of degree larger than $6$, the present approach does not apply: The locus
of binary forms with extra symmetries has codimension at least $2$ as soon as the
degree exceeds $6$, and hence does not contain a divisor $Z(F)$.
We remark that the invariant ring for binary forms of degree $8$ has been
explicitly described by Shioda \cite{shioda67}, and its structure
is indeed more complicated than \eqref{eq:typicalR}. Beyond those examples treated here,
the only cases known to the author that can 
be studied with similar methods are the actions of
finite subgroups of $\SL(2)$ on $\PP^1$ (with the natural
linearization given by the action of $\SL(2)$ on $\Aff^2$), whose invariant rings
have a structure close to that of \eqref{eq:typicalR}
\cite[Section~4.5]{springer77}.

The text roughly consists of two parts: In Sections
\ref{sec:root}, \ref{sec:rig} and \ref{sec:stackygit} we recall
standard stack theoretic notions (the root construction,
rigidification, the canonical smooth stack), and investigate
their meaning for the stacks arising from graded rings of the form \eqref{eq:typicalR}.
In Sections \ref{sec:sym},
\ref{sec:moduli} and \ref{sec:cubics} we analyse the stacky GIT
quotients corresponding
to the classically studied actions of $\SL(n+1)$ on hypersurfaces
in $\PP^n$. The material in this second part has a classic taste,
and is undoubtedly well known. I claim originality only for the
interpretation of these results in terms of the stacky GIT
quotient.
It should also be remarked that the stacky GIT quotient
for binary sextics, and its ``memory'' of $Z(F)$, 
has been considered by Hassett \cite[Section~3.1]{hassett2005}.

For an overview of classical invariant theory, and for more
detailed references to original works than is given here, the
reader is referred to the book by Dolgachev \cite{dolgachev2003},
which has been very useful in preparing this text. I learnt the
right language (the root construction, etc.)~for these
investigations from a talk on toric stacks by B.~Fantechi at the
Institut Mittag-Leffler in May, 2007.

\section{Notation}

We work over an algebraically closed field $k$ of characteristic
zero.
Following Fantechi et.~al.~\cite{FMN2007}, we define a \emph{DM stack} to be a separated
Deligne-Mumford stack.

Let $R = \bigoplus_{d\ge 0} R_d$ be a nonnegatively graded
$k$-algebra with $R_0=k$, and let $R_+$ be the maximal ideal generated by
elements of strictly positive degree. Thus $\Spec(R)$ is a cone,
with vertex $0\in \Spec(R)$ defined by $R_+$.
We write
\begin{equation}\label{eq:stackyproj}
\PProj(R) = [(\Spec(R)\setmin \{0\})) / \Gm]
\end{equation}
for the stack quotient by the natural action of $\Gm$
associated to the grading. The coarse space of $\stack{X}=\PProj(R)$
is the usual scheme $X=\Proj(R)$.
Note that line bundles
on $\stack{X}$ can be identified with $\Gm$-linearized line
bundles on $\Spec(R)\setmin \{0\}$. 
Thus, the graded $R$-module $R(n)$ gives rise to a line bundle
$\OO_{\stack{X}}(n)$ on $\stack{X}$, although the sheaf $\OO_X(n)$
on $X$ may fail to be locally free.

\begin{example}
Let $d_1,\dots,d_n$ be positive integers, and let $k[t_1,\dots,t_n]$ denote
the weighted polynomial ring in which $t_i$ has degree $d_i$. Then the
\emph{weighted projective stack} with the given weights is defined as
\begin{equation*}
\stack{P}(d_1,\dots,d_n) = \PProj(k[t_1,\dots,t_n])
\end{equation*}
and its coarse space is the usual weighted projective space
\begin{equation*}
\PP(d_1,\dots,d_n) = \Proj(k[t_1,\dots,t_n]).
\end{equation*}
\end{example}

\begin{definition}
Let $G$ be a reductive group acting on a projective scheme
$Y\subset\PP^N$ via a linear representation. Let $S$ be the
homogeneous coordinate ring of $Y$. Then the stack
\begin{equation*}
\stack{X} = \PProj(S^G)
\end{equation*}
is the \emph{stacky GIT
quotient} of the linearized action of $G$ on $Y$.
\end{definition}

If $f$ is a homogeneous element in $R=S^G$ of degree $r\ne 0$, the
ring $R/(f-1)$ is $\ZZ/(r)$-graded, and there is a corresponding
action of the cyclic group $\runity_r = \Spec k[t]/(t^r-1)$ on its
spectrum. The stack quotient $[\Spec(R/(f-1))/\runity_r]$ is an open
substack of $\PProj(R)$, and for $f$ running through a generator
set of $R$, these open substacks form an open cover. Thus the
stacky GIT quotient is a
DM stack
with cyclic
automorphism groups.

\section{Root stacks}\label{sec:root}

We fix a
DM stack
$\stack{X}$, a line bundle $\sheaf{L}$ on $\stack{X}$
with a global section $s$, and a natural number $r$. Associated to
these data, there is a canonically defined stack
\begin{equation*}
\pi\colon \stack{X}[\sqrt[r]{s}] \to \stack{X}
\end{equation*}
over $\stack{X}$, called
the $r$'th root along $s$.

Briefly, $\stack{X}[\sqrt[r]{s}]$ is obtained from $\stack{X}$ by
adding $\runity_r$ to the automorphism groups along the vanishing
locus of $s$, enabling one to extract an $r$'th root of $\pi^*(s)$.
Away from the vanishing locus of $s$, the map $\pi$ is an
isomorphism.

More precisely, an object of $\stack{X}[\sqrt[r]{s}]$ over a scheme $T$
consists of a map $f\colon T\to\stack{X}$, together with a line
bundle $\sheaf{M}$ on $T$ with a global section $t$, and an
isomorphism
\begin{equation*}
\sheaf{M}^r\iso f^*(\sheaf{L})
\end{equation*}
sending $t^r$ to $s$. The foundations of this construction can be
found in a paper by Cadman \cite{cadman2007}.
In particular, Cadman
shows that the root construction applied to a Deligne-Mumford stack
is again Deligne-Mumford.

\begin{example}
Let $X=\Spec(R)$ be an affine scheme, and let $s\in R$, considered as
a section of the trivial line bundle. Then the $r$'th root stack
along $s$ is the stack quotient
\begin{equation*}
X[\sqrt[r]{s}] = [\Spec \left(R[t]/(t^r-s)\right) /\runity_r]
\end{equation*}
where the $\runity_r$-action corresponds to the canonical
$\ZZ/(r)$-grading of $R[t]/(t^r-s)$.
\end{example}

Our aim is to establish a graded analogue of this example. To state
the result, we introduce the following notation: If $R=\bigoplus_{d\ge
0}R_d$ is a graded ring and $n$ is a natural number, let
$R^{(1/n)}$ be the same ring with grading defined by
declaring that $d\in R_n$ has degree $dn$ in $R^{(1/n)}$. Note
that $R^{(1/n)}_d=0$ unless $n$ divides $d$. In the following we
use the notation $\PProj(R)$ for the stack \eqref{eq:stackyproj}.

\begin{lemma}\label{lem:root}
Let $R=\bigoplus_{n\ge 0} R_n$ be a graded $k$-algebra with
$R_0=k$, and let $\stack{X}=\PProj(R)$.
Let $s\in R_n$ be a homogeneous element, considered as a global
section of $\OO_{\stack{X}}(n)$. Let $r$ be a natural number, and
assume that $r$ and $n$ have no common factors. Then the $r$'th root
stack of $\stack{X}$ along $s$ is
\begin{equation*}
\stack{X}[\sqrt[r]{s}] = \PProj(S)
\qquad\text{where}\qquad
S = R^{(1/n)}[t]/(t^r-s)
\end{equation*}
with grading defined by letting $t$ have degree $n$.
\end{lemma}

The lemma fails without the condition that $r$ and $n$ have no
common factors, as the
following example shows.

\begin{example}
Let $R=k[x_0,\dots,x_n]$ where the $x_i$'s have degree $1$. Then
$\stack{X}$ is the scheme $\PP^n$. Consider the square root
stack along a quadratic hypersurface, so let $s\in R$ have degree
$2$. A point on the square root stack has automorphism group
$\runity_2$ if it belongs to the vanishing locus of $s$; otherwise
its automorphism group is trivial.
On the other hand
\begin{equation*}
S = k[x_0,\dots,x_n,t]/(t^2-s).
\end{equation*}
The grading defined in the lemma is such that the $x_i$'s and $t$
all have degree $2$. But then $\PProj(S)$ has
$\runity_2$ as automorphism group everywhere, and is thus not
the square root stack along $s$. The only other sensible grading
on $S$ is that in which $x_i$ and $t$ have degree
$1$, but then $\PProj(S)$ would be a scheme, and we still do
not get the square root stack.
\end{example}

\begin{proof}[Proof of Lemma \ref{lem:root}]
Let $X=\Spec(R)\setmin \{0\}$ and $Y=\Spec(S)\setmin \{0\}$, equipped
with $\Gm$-actions
\begin{align*}
\sigma_X&\colon \Gm\times X \to X\\
\sigma_Y&\colon \Gm\times Y \to Y.
\end{align*}
Viewing $Y$ as the subscheme of $X\times\Aff^1$ defined by
$t^r=s$, the action $\sigma_Y$ is the restriction of the
action on $X\times\Aff^1$ given on $T$-valued points by
\begin{equation}\label{eq:action}
(x,a) \mapsto (\sigma_X(\xi^r, x), \xi^na)
\end{equation}
for $(x,a)\in X(T)\times\Aff^1(T)$ and $\xi\in\Gm(T)$.

The claim is that $\stack{Y}=[Y/\Gm]$ is the $r$'th root stack of
$\stack{X}=[X/\Gm]$ along $s$. We will show how to map objects in
$\stack{X}[\sqrt[r]{s}]$ over a scheme $T$ to objects in
$\stack{Y}$ over $T$ and conversely, leaving out the straight
forward verification that these maps are quasi-inverse functors
in a natural way.
With reference to the diagram
\begin{equation}\label{dia:torsors}
\begin{diagram}[height=1.5em]
Q & \rTo^{(g,u)} & &  X\times\Aff^1 \\
  & \rdTo^{\pi}  & & & \rdTo \\
\dTo^q     &  & P & \rTo^f & & X \\
& \ldTo_p \\
T
\end{diagram}
\end{equation}
the objects in question are given by the following data:
\begin{enumerate}
\item
An object in $\stack{Y}$ over $S$ is a $\Gm$-torsor $q\colon Q\to
T$ together with a $\Gm$-equivariant map $Q\to Y$. Viewing $Y$ as
a subscheme of $X\times\Aff^1$, the latter becomes a pair $(g,u)$
as in the upper part of diagram \eqref{dia:torsors}, which is
equivariant with respect to the action \eqref{eq:action} on the
target.
\item
An object in $\stack{X}[\sqrt[r]{s}]$ over $T$ is a $\Gm$-torsor
$p\colon P\to T$ together with a $\Gm$-equivariant map $f$ as in
the lower part of diagram \eqref{dia:torsors}, a $\Gm$-linearized
line bundle $L$ over $P$ with an equivariant section $v\in
\Gamma^\Gm(P,L)$ and a $\Gm$-equivariant isomorphism
\begin{equation*}
L^r \iso P\times\Aff^1_{(n)} \quad
(=f^*(X\times\Aff^1_{(n)}))
\end{equation*}
which identifies $v^r$ with $f^*(s)$. Here we write $\Aff^1_{(n)}$
for the affine line equipped with the $\Gm$-action of weight $n$.
\end{enumerate}

Given data (1), let $P=Q/\runity_r$ and let $p\colon P\to T$ be the map
induced by $q$. This is a $\Gm$-torsor with respect to the
induced action of $\Gm/\runity_r\iso\Gm$. Moreover, $g$ induces
an equivariant map $f$ making diagram \eqref{dia:torsors} commute. On
$P$ there is the $\Gm$-linearized line bundle
\begin{equation*}
L = (Q\times\Aff^1_{(n)}) / \runity_r
\end{equation*}
with a section $v\in \Gamma^\Gm(P, L)$ induced by the unity
section of $Q\times\Aff^1_{(n)}$, and a canonical trivialization
\begin{equation*}
L^r \iso (Q\times\Aff^1_{(n)})^r/\runity_r
\iso P\times\Aff^1_{(n)}.
\end{equation*}
This defines data as in (2).

Conversely, let data (2) be given. The line bundle $L$ with the
trivialization of $L^r$ gives rise to a
$\runity_r$-torsor $\pi\colon Q\to P$, defined as follows:
Identifying $(\dual{L})^r$ with the trivial line
bundle, we let $Q\subset \dual{L}$ be the $r$'th roots of unity
in each fibre. This is clearly a $\runity_r$-torsor under the
action of multiplication in the fibres. Now we define a new
$\runity_r$-action on $Q$ by letting $\xi\in\runity_r$ act by
multiplication with $\xi^n$ in the fibres. Since $r$ and $n$ are
relatively prime, the $n$'th power endomorphism on $\Gm$ induces
an automorphism on $\runity_r$, so $Q$ is a $\runity_r$-torsor also
under this new action. Moreover, it extends to a
$\Gm$-action as follows: Let $\xi\in\Gm$ act on $\dual{L}$ by
composing the contragradient action of $\xi^r$ (using the given $\Gm$-action
on $L$) with multiplication by $\xi^n$ in the fibres:
\begin{equation*}
\begin{diagram}[height=4ex,width=6em,tight]
\Gm \times \dual{L} & \rTo^{(n,r)\times 1_{\dual{L}}} &
\Gm\times\Gm\times\dual{L} \\
& \rdTo(2,4) & \dTo_{1_\Gm} \times(\text{contragrad.})\\
& & \Gm\times\dual{L}\\
& & \dTo_{(\text{scalar mult.})}\\
& & \dual{L}
\end{diagram}
\end{equation*}
Then $Q\subset\dual{L}$ is $\Gm$-invariant, and the induced
action extends the $\runity_r$-action defined above. The
given $\Gm$-action on $P$ agrees with the induced action of
$\Gm/\runity_r\iso\Gm$ on $Q/\runity_r\iso P$, and it follows
that $Q$ is a $\Gm$-torsor over $T$. We now let $g=f\circ\pi$ and
let $u$ be the restriction of
\begin{equation*}
\dual{v}\colon\dual{L}\to \Aff^1
\end{equation*}
to $Q$. This defines data as in (1).
\end{proof}

\section{Rigidification}\label{sec:rig}

Intuitively, the rigidification of a given stack is ``the same stack'' with
the general automorphism group removed.

\begin{definition}\label{def:rig}
The \emph{rigidification} of an irreducible
DM stack
$\stack{X}$
is a dominant map
\begin{equation*}
f\colon \stack{X} \to \rig{\stack{X}}
\end{equation*}
to another
DM stack
$\rig{\stack{X}}$, such that the automorphism group $\Aut(x)$ of a general point $x\in\rig{\stack{X}}$ is trivial, and such that $f$ is universal with this property.
\end{definition}

We spell out the meaning of universality: For every dominant map $g\colon
\stack{X} \to \stack{Y}$, to a
DM stack
$\stack{Y}$ whose general points have
trivial automorphism groups, we require the existence of a map $h\colon
\rig{\stack{X}}\to\stack{Y}$ making the diagram
\begin{equation*}
\begin{diagram}[height=5ex,width=3em,tight]
&\stack{X} \\
&\dTo^f &\rdTo^g \\
&\rig{\stack{X}} &\rTo^h &\stack{Y}
\end{diagram}
\end{equation*}
2-commutative, and the map $h$ is unique up to unique natural equivalence.

\begin{remark}
Rigidifications are defined in the literature in greater generality
\cite{ACV2003, AOV2008}. Namely, one chooses a subgroup stack $G$ of the
inertia stack $I(\stack{X})$, and defines the
rigidification with respect to $G$ to be a stack receiving a map from
$\stack{X}$, with automorphisms belonging to $G$ being killed, and universal with
this property.  The rigidification of Definition \ref{def:rig} is the special
case where $G$ is taken to be the closure of the
union of the automorphism groups of general points $x$ of $\stack{X}$.  The rigidification
in this sense is known to exist \cite[Example~A.3]{AOV2008} under general
conditions. In order to keep the presentation self contained we give a
direct construction in Proposition \ref{prop:rig} in the situation we need here.
\end{remark}

Without any conditions on $\stack{X}$, the rigidification does
not necessarily exist. We treat an easy special case where it
does exist:
Consider a DM stack of the form $[X/G]$
for an algebraic
group $G$ acting on an irreducible scheme $X$,
and suppose that the general
stabilizer group equals the common stabilizer group 
\begin{equation*}
H = \left\{ g\in G \suchthat gx=x \;\forall\;x\in X\right\}.
\end{equation*}
More precisely, we assume that there is an open subset
$U\subseteq X$ such that $G_x=H$ for all $x\in U$. Note that
$H$ is normal, and the induced action of $G'=G/H$ has
generically trivial stabilizer.

\begin{proposition}\label{prop:rig}
Let $G$ be an algebraic group scheme acting on an irreducible
scheme $X$,
such that $[X/G]$ is a DM stack and
such that
the general stabilizer group equals the
common stabilizer group $H\subseteq G$.  Then the stack
quotient $[X/G]$ admits a rigidification, in fact
\begin{equation*}
\rig{[X/G]}\iso[X/G']
\end{equation*}
where $G'=G/H$.
\end{proposition}

Before proving the proposition, we establish a lemma.

\begin{lemma}
Let $\stack{Y}$ be a
DM stack
containing an algebraic
space $U\subseteq\stack{Y}$ as an open dense substack. Let
\begin{equation*}
\alpha,\beta\colon S\rightrightarrows \stack{Y}
\end{equation*}
be two maps from a scheme $S$, such that the restrictions of $\alpha$ and
$\beta$ to every component of $S$ are dominant. Then equivalence
$\alpha\iso\beta$ is a local property on $S$. Precisely, if
$p\colon T\to S$ is a surjective, flat map, locally of finite
presentation, such that
$p^*(\alpha)$ and $p^*(\beta)$ are equivalent, then already
$\alpha$ and $\beta$ are
equivalent.
\end{lemma}

\begin{proof}
There is an open dense subset $S'\subset S$ whose image under
both $\alpha$ and $\beta$ is contained in $U$.
Since the restricted maps $\res{\alpha}{S'}$ and
$\res{\beta}{S'}$ have the algebraic space $U$ as codomain, it is clear
that equivalence $\res{\alpha}{S'}\iso\res{\beta}{S'}$ is a local property
on $S'$, and so holds by assumption.

Recall that, according to our conventions, the DM stack $\stack{Y}$ is
separated.
Thus there is a
closed subscheme $S''\subseteq S$ which is universal with the
property that the restrictions of $\alpha$ and $\beta$ to $S''$ are
equivalent. We have established that $S''$ contains $S'$, which is
dense in $S$. Thus $S''=S$ and we conclude that $\alpha$
and $\beta$
are equivalent.
\end{proof}

\begin{proof}[Proof of the proposition]
Clearly, the quotient stack $[X/G']$ has generically trivial
stabilizer, and admits a canonical surjective map
\begin{equation*}
f\colon [X/G]\to [X/G'].
\end{equation*}
Let $g\colon [X/G]\to \stack{Y}$ be another dominant map to a
DM stack
$\stack{Y}$ with generically trivial automorphism
groups, and consider the diagram
\begin{equation*}
\begin{diagram}[size=3em]
          &                  &            &                                          & X\\
G\times X & \rTo             & G'\times X & \ruTo(2,1)_{\sigma'} \ruTo(4,1)^{\sigma} &   & \rdTo(2,1)^{\pi} & [X/G] & \rTo^g &\stack{Y}\\
          & \rdTo(4,1)_{p_2} &            & \rdTo(2,1)^{p_2}                         & X & \ruTo(2,1)_{\pi}
\end{diagram}
\end{equation*}
where $\sigma$ and $\sigma'$ are the actions, $p_2$ is second
projection and $\pi$ is the quotient map. We claim that
$g\circ\pi\circ\sigma'$ and $g\circ\pi\circ p_2$ are
equivalent maps $G'\times X\to\stack{Y}$. Once this is
established, it follows from the universality of the quotient
stack $[X/G']$ that $g$ factors through $f$ as required.

We apply the lemma as follows: Let $U\subset\stack{Y}$ be an open
dense substack with trivial automorphism groups everywhere. It is an
algebraic space. By definition of the quotient $[X/G]$, the two
maps $\pi\circ\sigma$ and $\pi\circ p_2$ are equivalent. Hence
also their compositions with $g$ are equivalent.
Taking $R\to S$ in the lemma to be the flat surjective map $G\times
X\to G'\times X$, we find that the two maps from $G'\times X$ to
$\sheaf{Y}$ are already equivalent, as claimed.
\end{proof}

\begin{corollary}\label{cor:projrig}
Let $R = \bigoplus_{d\ge0} R_d$ be a graded $k$-algebra with
$R_0=k$, whose nilradical
is prime (i.e.~$\Spec(R)$ is irreducible). Then $\PProj(R)$
admits a rigidification. In fact, letting \begin{equation*}
n = \hcf\{d \suchthat R_d \ne 0\},
\end{equation*}
we have
\begin{equation*}
\rig{\PProj(R)} \iso \PProj(R^{(n)}).
\end{equation*}
\end{corollary}

\begin{proof}
With $n$ as above, the subgroup $\runity_n\subset\Gm$ acts
trivially on $\Spec(R)$.  On the other hand, there exists a
finite set of homogeneous
elements $f_1,\dots,f_r$ in $R$ such that the highest common
factor of their degrees equals $n$. Let $U\subset\Spec(R)$ be
the open set defined by the simultaneous nonvanishing of the
$f_i$'s. Then the stabilizer group of any point in $U$ is
exactly $\runity_n$.
Thus $\runity_n$ is both the common stabilizer group and
the generic stabilizer group, so the proposition applies. The quotient $\Gm/\runity_n$ is again isomorphic to
$\Gm$, and the induced action corresponds to the grading in which $f\in R_d$ is
given degree $d/n$. This is by definition the grading on $R^{(n)}$.
\end{proof}

\begin{remark}\label{rem:gerbe}
In the last Corollary, 
$R^{(n)}$
and $R$ are essentially the same ring, only with a 
different grading.
Geometrically, this can be phrased as follows:
The stack $\PProj(R)$ is
the $n$'th root stack of $\OO(1)$ on
$\PProj(R^{(n)})$, defined similarly as the root stack in Section \ref{sec:root},
only without the section $s$ (this construction can also be found in Cadman's
paper \cite{cadman2007}). At the level of points with automorphisms groups, $\PProj(R)$
is obtained from $\PProj(R^{(n)})$ by sticking in an extra automorphism group $\runity_n$
everywhere. More precisely,
\begin{equation*}
\PProj(R)\to \PProj(R^{(n)})
\end{equation*}
is an \emph{essentially trivial
$\runity_n$-gerbe}. We refer the reader to Lieblich \cite{lieblich2007} and
Fantechi et.~al. \cite{FMN2007} for systematic expositions, but mention
briefly that a $\runity_n$-gerbe over a stack $\stack{X}$ corresponds to an
element of $H^2(\stack{X},\runity_n)$, and is called essentially trivial if the
push forward to $H^2(\stack{X},\Gm)$ vanishes. This is equivalent
\cite[Proposition 2.3.4.4]{lieblich2007} \cite[Remark 6.4]{FMN2007} to the
statement that the gerbe is the $n$'th root stack of a line bundle on
$\stack{X}$.
\end{remark}

\section{The stacky GIT quotient}\label{sec:stackygit}

The invariant ring of binary quartics, and that of cubic plane
curves, are weighted polynomial rings in two variables (see
Sections \ref{sec:moduliquartics} and \ref{sec:cubiccurves}). In
this section we study the more interesting invariant rings for binary quintics, binary
sextics, and cubic surfaces: All of these have the structure
(see Sections 
\ref{sec:moduliquintics}, \ref{sec:modulisextics} and \ref{sec:cubicsurf})
\begin{equation}\label{eq:typicalR}
R \iso k[t_1,\dots,t_{n+1}]/(t_{n+1}^2 - F(t_1,\dots,t_n))
\end{equation}
where $t_i$ are homogeneous generators of some positive weight
$d_i$, and $F$ is a weighted homogeneous polynomial of degree
$2d_{n+1}$. The weights fulfil the following three conditions:
\begin{enumerate}
\item[(i)] The highest common factor $d$ of $d_1,\dots,d_n$ does not
divide $d_{n+1}$.
\item[(ii)] The weights $e_i=d_i/d$, for $i\le n$, are well formed, i.e.
no $n-1$ among them have a common factor.
\item[(iii)] $2d_{n+1}/d$ is even.
\end{enumerate}
The first condition says that $t_{n+1}$ is not an element of the
subring $R^{(d)}\subseteq R$, which thus is a weighted polynomial
ring. The second condition says that its generators $t_1,\dots,t_n$
have well formed weights. The last condition says that the
degree of $F$, as an element of $R^{(d)}$, is even, which is the
condition needed to apply Lemma \ref{lem:root} to extract a
square root.

Recall \cite[Section 4.1]{FMN2007} that any variety $X$, with at worst finite quotient
singularities, is in a canonical way the coarse space of a smooth
DM stack,
the \emph{canonical stack} $\can{X}$.
More precisely, for any
smooth DM stack
$\stack{Y}$ with $X$ as coarse space, there is a unique map $\stack{Y}\to\can{X}$,
compatible with the maps to $X$.
Thus, if
$\stack{X}$ is a
DM stack admitting a smooth rigidification,
having a variety
$X$ with finite quotient singularities as coarse space,
then the universal properties of the rigidification and the canonical stack yield a
factorization of the canonical map $\stack{X}\to X$ as
\begin{equation*}
\stack{X} \to \rig{\stack{X}} \to \can{X} \to X.
\end{equation*}

\begin{theorem}\label{thm}
Let $R$ be a graded $k$-algebra of the form \eqref{eq:typicalR},
satisfying conditions (i), (ii)
and (iii) above. Let $\stack{X}$ be
the stack $\PProj(R)$ and let $X$ be its coarse space $\Proj(R)$.
\begin{enumerate}
\item $X$ is the weighted projective space $\PP(e_1,\dots,e_n)$,
and its canonical stack is the weighted projective stack $\can{X} = \stack{P}(e_1,\dots,e_n) = \PProj(R^{(d)})$.
\item The rigidification of $\stack{X}$ is
$\rig{\stack{X}} = \PProj(R^{(d/2)})$.
\item The map $\rig{\stack{X}}\to
\can{X}$ is the square root along $F$, considered as a
section of $\OO_{\can{X}}(2d_{n+1}/d)$.
\end{enumerate}
\end{theorem}

\begin{proof}
The coarse moduli space of the stack $\PProj(R)$ is the scheme
$\Proj(R)$. Since $\Proj(R)\iso\Proj(R^{(d)})$, and
\begin{equation*}
R^{(d)} \iso k[t_1,\dots,t_n]
\end{equation*}
is a weighted polynomial ring, where each generator $t_i$ has
degree $e_i=d_i/d$, it is clear that the coarse moduli space is
the weighted projective space as claimed.
It is a standard fact \cite[Example 7.25]{FMN2007}
that its canonical smooth stack is $\stack{P}(e_1,\dots,e_n) =
\PProj(R^{(d)})$, using that the weights are well
formed. This proves (1).

Next we apply Corollary \ref{cor:projrig}. Since $2d_{n+1} = \deg F$
and $d$ divides $\deg F$, we see that $d$ is even and the highest
common factor of $d_1,\dots,d_{n+1}$ is $d/2$ (using the
assumption that $d$ does not divide $d_{n+1}$). This proves (2).

Finally, Lemma \ref{lem:root} immediately gives
$\PProj(R^{(d/2)})$ as the square root stack of $\PProj(R^{(d)})$
along $F$. This proves (3).
\end{proof}

\begin{remark}
The theorem tells us in particular that the stack $\stack{X}$
remembers not only its coarse moduli space $X$, but also the
divisor defined by $F$. Conversely, knowing $X$ and $F$, we can
reconstruct the rigidification of $\stack{X}$ by extracting a
square root of $F$ on the canonical stack associated to $X$.
Finally, $\stack{X}$ is an essentially trivial
$\runity_{(d/2)}$-gerbe over its rigidification, as in Remark
\ref{rem:gerbe}.
\end{remark}

\section{Symmetries of binary forms}\label{sec:sym}

The aim of this section is to survey Klein's classification
\cite{klein1884} of binary forms according to
their symmetries, i.e.~their stabilizer groups. Throughout, we
identify forms that differ by a nonzero scalar factor. Thus,
by the stabilizer group of a binary form $f$, we shall mean the
elements of $\SL(2)$ under which $f$
is semi-invariant, i.e.~invariant up to a nonzero scalar factor.

Recall that a binary form of degree $d$ is stable if and only if
all its roots have multiplicity strictly less than $d/2$. Such a binary
form has finite stabilizer group in $\SL(2)$. More generally, any binary form with
at least three distinct zeros has finite stabilizer group. This leaves
just the case $f=x^ny^m$ (modulo $\SL(2)$ and scale), whose stabilizer
group consists of all diagonal matrices in $\SL(2)$ if $n\ne m$, and all diagonal
and antidiagonal matrices if $n=m$. From now on we assume that $f$
is a binary form with finite stabilizer group $G\subset\SL(2)$.
Hence $G$ is a cyclic, dihedral,
tetrahedral, octahedral or icosahedral group, by the well known classification
of finite subgroups of $\SL(2)$. More precisely, a conjugate of $G$ equals one
of the subgroups listed in Table \ref{table:generators}. The conjugation
corresponds to picking another representative for the orbit of $f$ under
$\SL(2)$, since $\gamma G\gamma^{-1}$ is the stabilizer group of $\gamma f$.

Thus, to classify binary forms with finite stabilizer group, it
suffices to determine the semi-invariant forms for each group $G$ in
Table \ref{table:generators}. A general $G$-orbit in $\PP^1$ has
degree $|G|/2$. For non-cyclic $G$, there are precisely
three special orbits of smaller degree, defined by the
vanishing of three so called ground forms $F_1$, $F_2$, and
$F_3$. These are listed in Table \ref{table:groundforms}. We put
$\nu_i = |G|/(2\deg F_i)$.

\begin{table}
\begin{tabular}{l|ll}
Group   & Generators
\\ \hline\\[-2ex]
$C_n$ &
$
\left(
\begin{matrix}
\epsilon & 0\\
0 & \epsilon^{-1}
\end{matrix}
\right)$
& $\text{($\epsilon = \sqrt[2n]{1}$ primitive)}$
\\[1em] \hline\\[-2ex]
$D_n$
&
$
\left(
\begin{matrix}
\epsilon & 0\\
0 & \epsilon^{-1}
\end{matrix}
\right),
\left(
\begin{matrix}
0 & i\\
i & 0
\end{matrix}
\right)$
& $\text{($\epsilon = \sqrt[2n]{1}$ primitive)}$
\\[1em] \hline\\[-2ex]
$T$ &
\multicolumn{2}{l}{
$
\left(
\begin{matrix}
i & 0 \\
0 & -i
\end{matrix}
\right),
\left(
\begin{matrix}
0 & i \\
i & 0
\end{matrix}
\right),
\frac{1}{2}\left(
\begin{matrix}
1+i & -1+i \\
1+i & 1-i
\end{matrix}
\right)
$}
\\[1em] \hline\\[-2ex]
$O$ &
\multicolumn{2}{l}{
$
\left(
\begin{matrix}
i & 0 \\
0 & -i
\end{matrix}
\right),
\left(
\begin{matrix}
0 & i \\
i & 0
\end{matrix}
\right),
\frac{1}{2}\left(
\begin{matrix}
1+i & -1+i \\
1+i & 1-i
\end{matrix}
\right),
\frac{1}{\sqrt{2}}\left(
\begin{matrix}
1+i & 0\\
0 & 1-i
\end{matrix}
\right)
$}
\\[1em] \hline\\[-2ex]
$I$ &
$
\left(
\begin{matrix}
\epsilon^3 & 0\\
0 & \epsilon^2
\end{matrix}
\right),
\frac{1}{\sqrt{5}}\left(
\begin{matrix}
\epsilon-\epsilon^4 & \epsilon^3-\epsilon^2 \\
\epsilon^3-\epsilon^2 & \epsilon^4-\epsilon
\end{matrix}
\right)$
& $\text{($\epsilon = \sqrt[5]{1}$ primitive)}$
\end{tabular}
\vspace{1ex}
\caption{Polyhedral groups}\label{table:generators}
\end{table}

\begin{table}
\begin{tabular}{l|ll}
Group   & Ground forms
\\ \hline\\[-2ex]
$D_n$
&
$
\begin{array}{l}
F_1 = x^n+y^n \\
F_2 = x^n-y^n \\
F_3 = xy
\end{array}
$
\\ \hline\\[-2ex]
$T$
&
$
\begin{array}{l}
F_1 = x^4+2\sqrt{-3}x^2y^2+y^4 \\
F_2 = x^4-2\sqrt{-3}x^2y^2+y^4 \\
F_3 = xy(x^4-y^4)
\end{array}
$
\\ \hline\\[-2ex]
$O$
&
$
\begin{array}{l}
F_1 = xy(x^4-y^4) \\
F_2 = x^8+14x^4y^4+y^8 \\
F_3 = x^{12}-33x^8y^4-33x^4y^8+y^{12}
\end{array}
$
\\ \hline\\[-2ex]
$I$
&
$
\begin{array}{l}
F_1 = xy(x^{10}+11x^5y^5-y^{10}) \\
F_2 = -(x^{20}+y^{20})+228(x^{15}y^5-x^5y^{15}) - 494x^{10}y^{10} \\
F_3 = (x^{30}+y^{30}) + 522(x^{25}y^5-x^5y^{25}) - 10005(x^{20}y^{10}+x^{10}y^{20})
\end{array}
$
\end{tabular}
\vspace{1ex}
\caption{Ground forms}\label{table:groundforms}
\end{table}

\begin{lemma}[Klein \cite{klein1884}]\label{lem:klein}
A binary form $f$ is semi-invariant under the cyclic group $C_n$ if and only if (up to a scalar factor)
\begin{equation*}
f = x^\alpha y^\beta \prod_{i=1}^N (\lambda_i x^n + \mu_i y^n)
\end{equation*}
where $\alpha$, $\beta$ and $N$ are nonnegative integers, and $(\lambda_i:\mu_i)\in\PP^1$ are parameters.

A binary form $f$ is semi-invariant under one of the groups $D_n$, $T$, $O$,
$I$ if and
only if (up to a scalar factor)
\begin{equation*}
f = F_1^\alpha F_2^\beta F_3^\gamma\prod_{i=1}^N(\lambda_iF_1^{\nu_1}+\mu_iF_2^{\nu_2})
\end{equation*}
where $\alpha$, $\beta$, $\gamma$ and $N$ are nonnegative integers,
$(\lambda_i:\mu_i)\in\PP^1$ are parameters, and $F_1$, $F_2$ and $F_3$ are the ground forms associated with the group.
\end{lemma}

\begin{remark}
Consider the central projection of an inscribed regular
polyhedron onto the Riemann sphere $\PP^1$ (over $\CC$).
Then the ground forms of the group corresponding to the
polyhedron have as zero loci the vertices,
the midpoints of the
faces and the midpoints of the edges,
respectively.
A similar statement applies
to the cyclic and dihedral groups. The generic orbits, on the other hand, are
the zero loci of the forms
$\lambda F_1^{\nu_1} + \mu F_2^{\nu_2}$.
\end{remark}

For each fixed (and small) $d$, we can use the lemma to explicitly write down
all semi-invariants of degree $d$ for each group in Table
\ref{table:generators}.  This easily leads to the following classification of
binary forms of low degree, where we include only the stable cases, and write
$\Stab(f)\subset\SL(2)$ for the stabilizer subgroup of $f$.

\subsection{Binary quartics}\label{sec:symquartics}

Every stable, i.e.~square free, quartic is equivalent under the
$\SL(2)$-action to
\begin{equation*}
f = \lambda(x^2+y^2)^2 + \mu(x^2-y^2)^2
\end{equation*}
for some $(\lambda:\mu)$, and so has the dihedral group $D_2$
contained in its stabilizer. Modulo the $\SL(2)$-action, there are exactly two quartics with larger stabilizer group:
\begin{align*}
&\text{(I)}  & f &= x^4 + y^4 & \Stab(f)&=D_4 \\
&\text{(II)} & f &= x^4 + 2\sqrt{-3}x^2y^2 + y^4 & \Stab(f)&=T
\end{align*}

\subsection{Quintics}\label{sec:symquintics}

A stable quintic is one with at most double roots. All quintics
are stabilized by $C_1=\{\pm1\}$. Modulo $\SL(2)$, the
quintics with larger stabilizer group are the following:
\begin{align*}
&\text{(I)}   & f &= x(x^2+y^2)(\lambda x^2 + \mu y^2) & \Stab(f)
&=C_2\\
&\text{(II)}  & f &= x^2(x^3+y^3) & \Stab(f) &= C_3\\
&\text{(III)} & f &= x(x^4+y^4) & \Stab(f) &= C_4\\
&\text{(IV)}  & f &= xy(x^3+y^3) & \Stab(f) &= D_3\\
&\text{(V)}   & f &= x^5+y^5 & \Stab(f) &= D_5
\end{align*}
Here, the pair $(\lambda: \mu)\in\PP^1$ is a parameter assumed to have
generic value, so that the listed cases are disjoint.

\subsection{Sextics}\label{sec:symsextics}

A stable sextic is one with at most double roots. All sextics are
stabilized by $C_1=\{\pm1\}$. Modulo $\SL(2)$, the sextics with
larger stabilizer group are the following:
\begin{align*}
&\text{(I)}   & f &= (x^2+y^2)\prod_{i=1}^2(\lambda_i x^2 + \mu_i
y^2) & \Stab(f) &= C_2\\
&\text{(II)}  & f &= x(x^5+y^5) & \Stab(f) &= C_{5}\\
&\text{(III)} & f &= 
xy(\lambda (x^2+y^2)^2
+ \mu (x^2-y^2)^2)
& \Stab(f) &=  D_2\\
&\text{(IV)}  & f &=
\lambda (x^3+y^3)^2
+ \mu(x^3-y^3)^2
& \Stab(f) &= D_3\\
&\text{(V)}   & f &= x^6+y^6 & \Stab(f) &= D_6\\
&\text{(VI)}  & f &= xy(x^4-y^4) & \Stab(f) &= O\\
&\text{(VII)}  & f &= x^2y(x^3+y^3) & \Stab(f) &= C_3\\
&\text{(VIII)} & f &= x^2(x^4+y^4)  & \Stab(f) &= C_4
\end{align*}
Here again the parameters $\lambda$, $\mu$, $\lambda_i$, $\mu_i$ are
assumed to have generic values, so that the listed cases are
disjoint.

Bolza \cite{bolza1887} produced a list
of symmetry groups for square free sextics, which is equivalent
to the first six items in our list. As we will return to Bolza's
work in Section \ref{sec:modulisextics}, we remark that the labels (I)-(VI) we
are using agree with Bolza's.

\section{Moduli spaces of binary forms}\label{sec:moduli}

We let
\begin{equation*}
R\subset k[a_0,\dots,a_d]
\end{equation*}
be the invariant ring for the $\SL(2)$-action on degree $d$
binary forms
\begin{equation*}
f = a_0 x^d + a_1 x^{d-1}y + \dots a_d y^d.
\end{equation*}
In this section we apply the results from the previous sections to analyse the
geometry of $\PProj(R)$, for small values of $d$.
In particular, we describe
the loci of binary forms with prescribed symmetry group in terms of the
intrinsic geometry of the stacky GIT quotient.

For $d=4$, we find that quartics with extra symmetries show up as points
in the stacky GIT quotient with nontrivial automorphism groups.

For $d=5$ and $d=6$, the invariant ring $R$ has the form studied in Section
\ref{sec:stackygit}. Thus, from the stacky GIT quotient we obtain the usual GIT
quotient scheme $\Proj(R)$ together with the divisor $Z(F)$. We find that the binary
forms corresponding to singularities of the scheme $\Proj(R)$ have special symmetry groups,
but there are also loci of binary forms with symmetries that do not give rise
to singularities. However, these loci show up as $Z(F)$, its
singularities, or the singularities of its singular locus. Thus the knowledge
of $\Proj(R)$ together with the divisor $Z(F)$ suffices to enable a geometric
description of all loci of binary forms with prescribed symmetry group.

For the explicit computations needed in this section we rely on a
computer algebra system such as Singular \cite{singular}. Armed
with such a system, the calculations are straight forward, and we
only give the results.

\subsection{Binary quartics}\label{sec:moduliquartics}

The invariant ring $R$ for binary forms of degree $4$ is freely generated by
two homogeneous invariants
\begin{align*}
I_2 &= a_0 a_4 - 4 a_1 a_3 + 3 a_2^2\\
I_3 &= a_0 a_2 a_4 - a_0 a_3^2 + 2 a_1 a_2 a_3 - a_1^2 a_4 -
a_2^3,
\end{align*}
where the subscript $i$ of an invariant $I_i$ indicates its
degree. Thus $R=k[I_2,I_3]$ is a weighted polynomial ring, and
the GIT quotient is
\begin{equation*}
X = \Proj(R) \iso \PP^1.
\end{equation*}
In particular it is a homogeneous space, so, morally, it looks
the same at all points. Thus the geometry of the quotient does
not single out the
two points corresponding to the special quartics (I) and (II)
from Section \ref{sec:symquartics}. On the other hand, the stacky GIT
quotient
\begin{equation*}
\stack{X} = \PProj(R) = \stack{P}(2,3)
\end{equation*}
is a weighted projective line, in the stack sense, and the two
special points $(1:0)$ and $(0:1)$ corresponding to special
quartics are distinguished by having automorphism groups
$\runity_2$ and $\runity_3$, respectively.

In the language of the root construction, the stack $\stack{X}$
is obtained from its coarse space $X$ by extracting a square root
of $(1:0)$ and a cube root of $(0:1)$.

\subsection{Binary quintics}\label{sec:moduliquintics}

The invariant ring for binary quintics can be written
\begin{equation*}
R = k[I_4, I_8, I_{12}, I_{18}]/(I_{18}^2-F(I_4,I_8,I_{12}))
\end{equation*}
where the generators $I_i$ are homogeneous of degree $i$ and $F$
is (weighted) homogeneous of degree $36$.

Thus Theorem \ref{thm} applies: The GIT quotient scheme is the
weighted projective plane
\begin{equation*}
X = \Proj(R) = \PP(1,2,3)
\end{equation*}
and the canonical stack $\can{X}$ is the weighted projective stack
$\stack{P}(1,2,3)$.
The stacky GIT quotient
$\stack{X}=\PProj(R)$ is an essentially trivial $\runity_2$-gerbe
over its rigidification $\rig{\stack{X}}$, which is obtained from
$\can{X}$ by extracting a square root of $F$.

We note that the weighted projective plane $X$ has cyclic
quotient singularities at $(0:1:0)$ and $(0:0:1)$, and is
otherwise smooth. The stack $\stack{X}$ also remembers the
divisor defined by $F$, which we now analyse.

The generators $I_i$ for the invariant ring are not uniquely
defined. In the following we choose the generators given by Schur
\cite{schur68}. With this choice, we have\footnote{Schur does not
give the relation, but it can be found in Elliot's book
\cite{elliot1895}. Elliot's and Schur's invariants $I_i$ agree up to
scale.}
\begin{equation}\label{eq:quintic-F}
324F(I_4,I_8,I_{12}) =
-9I_4I_8^4-24I_8^3I_{12}+6I_4^2I_8^2I_{12}
+72I_4I_8I_{12}^2+144I_{12}^3-I_4^3I_{12}^2,
\end{equation}
which is irreducible and is singular at $(1:0:0)$ and
$(-3:3:3)$. Denoting the closures of the loci of special quintics with Roman
numerals (I)-(V), according to the list in Section
\ref{sec:symquintics}, we
have:
\begin{itemize}
\item (I) is the divisor $Z(F)$
\item (II) and (III) are the two singularities of $\PP(1,2,3)$
\item (IV) and (VI) are the two singularities of $Z(F)$
\end{itemize}
Moreover, the curve $Z(F)$ passes through (III) but avoids
(II). The situation is summarized in Figure \ref{fig:quintics}.

\begin{figure}
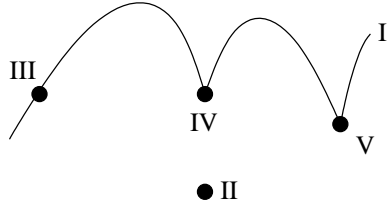

\input bin-quintics.pstex_t
\caption{Inside the moduli space of binary quintics.}\label{fig:quintics}
\end{figure}

\subsection{Binary sextics}\label{sec:modulisextics}

The invariant ring for binary sextics can be written
\begin{equation*}
R = k[I_2,I_4,I_6,I_{10},I_{15}]/(I_{15}^2 -
F(I_2,I_4,I_6,I_{10}))
\end{equation*}
where the generators $I_i$ are homogeneous of degree $i$ and $F$
is (weighted) homogeneous of degree $30$.

Thus Theorem \ref{thm} applies: The GIT quotient scheme is the
weighted projective space
\begin{equation*}
X = \Proj(R) = \PP(1,2,3,5)
\end{equation*}
and the canonical stack $\can{X}$ is the weighted projective stack
$\stack{P}(1,2,3,5)$.
The stacky GIT quotient $\stack{X} =
\PProj(R)$ is its own rigidification, and it is obtained from
$\can{X}$ by extracting a square root of $F$.

The weighted projective space $X$ has cyclic quotient
singularities at $(0:1:0:0)$, $(0:0:1:0)$ and $(0:0:0:1)$ and is
otherwise smooth. We next analyse the divisor $Z(F)$.

For a specific choice of generators $I_i$, Clebsch
\cite{clebsch1872} gives $F$ explicitly as twice the determinant
of the $3\times 3$ symmetric matrix $(a_{ij})$ with entries
\begin{align*}
a_{11} &= 2I_6 + \tfrac{1}{3}I_2I_4 & a_{23} &= \tfrac{1}{3}I_4(I_4^2+I_2I_6) + \tfrac{1}{3}I_6(2I_6+\tfrac{1}{3}I_2I_4)\\
a_{12} &= \tfrac{2}{3}(I_4^2+I_2I_6) & a_{33} &= \tfrac{1}{2}I_6I_{10}+\tfrac{2}{9}I_6(I_4^2+I_2I_6)\\
a_{13} &= I_{10} & a_{22} &= I_{10}.
\end{align*}
This polynomial $F$ is irreducible. Its zero locus $Z(F)$ is a
surface which is singular along a curve, having two components.
Each component has one singular point.

It turns out that these loci matches the
classification of sextics from Section \ref{sec:symsextics}: Again we
use roman numerals (I)-(VIII) for the closures of the loci in $X$
corresponding to the special sextics.
Equations for the loci (I)-(VI) were determined by Bolza
\cite{bolza1887}, and the points corresponding to the remaining
special sextics (VII) and (VIII) can be determined by evaluating
explicit expressions for the invariants $I_i$. The results are as
follows:
\begin{itemize}
\item (I) is the divisor $Z(F)$
\item (II), (VII) and (VIII) are the three singularities of
$\PP(1,2,3,5)$
\item (III) and (IV) are the two curves along which $Z(F)$ is
singular
\item (V) is the singular point of the curve (III)
\item (VI) is the singular point of the curve (IV)
\end{itemize}
Furthermore, the curves (III) and (IV) intersect in (V), (VI) and one additional point.
The latter corresponds to strictly semistable sextics,
i.e.~sextics with a triple root.

One also checks that the surface $Z(F)$ does not contain
(II) and (VII), but it contains (VIII) and is smooth there. The
situation is summarized in Figure \ref{fig:sextics}.

\begin{figure}
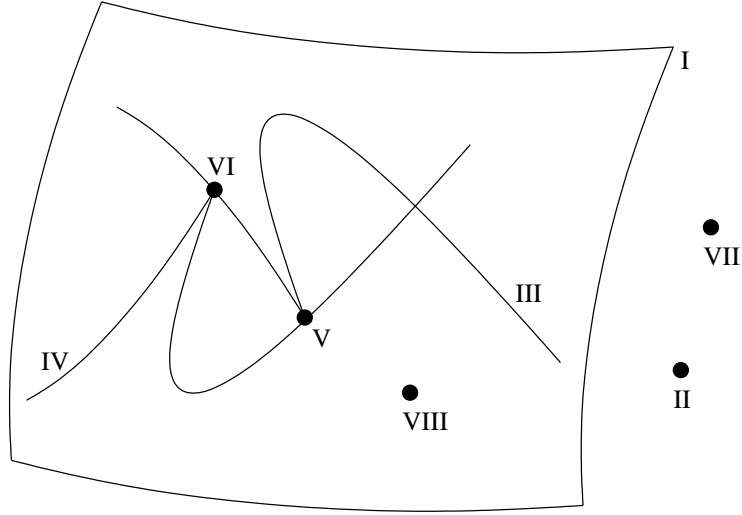

\input bin-sextics.pstex_t
\caption{Inside the moduli space of binary sextics.}\label{fig:sextics}
\end{figure}

\section{Cubic curves and surfaces}\label{sec:cubics}

In this section we describe the relation between the GIT quotient
scheme and the stacky GIT quotient corresponding to cubic plane
curves and cubic surfaces in space.

\subsection{Cubic plane curves}\label{sec:cubiccurves}

The invariant ring $R$ for the action of $\SL(3)$ on cubic forms
in three variables can be written
\begin{equation*}
R = k[I_4, I_6]
\end{equation*}
where the generators $I_i$ are homogeneous degree $i$ polynomials
in the coefficients of the cubic form. Thus the GIT quotient
scheme is $\PP^1$. By Corollary \ref{cor:projrig}, the stacky GIT
quotient $\PProj(R)$ is an essentially trivial $\runity_2$-gerbe
over its rigidification $\PProj(R^{(2)})$, which is the weighted
projective stack
$\stack{P}(2,3)$.
As in Section
\ref{sec:moduliquartics}, this stack is obtained from its
coarse space $\PP^1$ by extracting a square root of $(0:1)$ and a
cube root of $(1:0)$.

\subsection{Cubic surfaces}\label{sec:cubicsurf}

The invariant ring $R$ for the action of $\SL(4)$ on cubic forms
in four variables can be written
\begin{equation*}
R = k[I_8,I_{16},I_{24},I_{32},I_{40},I_{100}]/(I_{100}^2-F(I_8,I_{16},I_{24},I_{32},I_{40}))
\end{equation*}
where the generators $I_i$ are homogeneous of degree $i$ and $F$
is (weighted) homogeneous of degree $200$.

Thus Theorem \ref{thm}
applies: The GIT quotient scheme is the weighted projective
space
\begin{equation*}
X = \Proj(R) = \PP(1,2,3,4,5)
\end{equation*}
and the canonical stack $\can{X}$ is the weighted projective stack 
$\stack{P}(1,2,3,4,5)$.
The stacky GIT quotient
$\stack{X}=\PProj(R)$ is an essentially trivial $\runity_4$-gerbe over its
rigidification $\rig{\stack{X}}$, which is obtained from
$\can{X}$ by extracting a square root of $F$.

One may expect that the singularities of $X$, $Z(F)$, the
singularities of their singular loci etc.,
reflect a classification of cubic surfaces
according to their symmetries, as was the case for binary forms.
We have not investigated this further.

% The following was generated by bibtex and subsequently copied
% from stacky-git.bbl (as requested by Math. N.)


\begin{thebibliography}{99}

\bibitem{ACV2003}
D.~Abramovich, A.~Corti, and A.~Vistoli.
\newblock Twisted bundles and admissible covers.
\newblock {\em Comm. Algebra}, 31(8):3547--3618, 2003.
\newblock Special issue in honor of Steven L. Kleiman.

\bibitem{AOV2008}
D.~Abramovich, M.~Olsson, and A.~Vistoli.
\newblock Tame stacks in positive characteristic.
\newblock {\em Ann. Inst. Fourier (Grenoble)}, 58(4):1057--1091, 2008.

\bibitem{alper2008}
J.~Alper.
\newblock Good moduli spaces for {A}rtin stacks.
\newblock arXiv:0804.2242v2 [math.AG].

\bibitem{bolza1887}
O.~Bolza.
\newblock On {B}inary {S}extics with {L}inear {T}ransformations into
  {T}hemselves.
\newblock {\em Amer. J. Math.}, 10(1):47--70, 1887.

\bibitem{cadman2007}
C.~Cadman.
\newblock Using stacks to impose tangency conditions on curves.
\newblock {\em Amer. J. Math.}, 129:405--427, 2007.

\bibitem{clebsch1872}
A.~Clebsch.
\newblock {\em Theorie der bin\"aren algebraischen {F}ormen}.
\newblock Leipzig: Teubner, 1872.

\bibitem{dolgachev2003}
I.~Dolgachev.
\newblock {\em Lectures on invariant theory}, volume 296 of {\em London
  Mathematical Society Lecture Note Series}.
\newblock Cambridge University Press, Cambridge, 2003.

\bibitem{elliot1895}
E.~B. Elliot.
\newblock {\em An introduction to the Algebra of Quantics}.
\newblock Oxford University Press, 1895.

\bibitem{FMN2007}
B.~Fantechi, E.~Mann, and F.~Nironi.
\newblock {S}mooth toric {DM} stacks.
\newblock arXiv:0708.1254v1 [math.AG].

\bibitem{singular}
G.-M. Greuel, G.~Pfister, and H.~Sch\"onemann.
\newblock {\sc Singular} 3.0.
\newblock {A Computer Algebra System for Polynomial Computations}, Centre for
  Computer Algebra, University of Kaiserslautern, 2005.
\newblock {\tt http://www.singular.uni-kl.de}.

\bibitem{hassett2005}
B.~Hassett.
\newblock Classical and minimal models of the moduli space of curves of genus
  two.
\newblock In {\em Geometric methods in algebra and number theory}, volume 235
  of {\em Progr. Math.}, pages 169--192. Birkh\"auser Boston, Boston, MA, 2005.

\bibitem{klein1884}
F.~Klein.
\newblock {\em Vorlesungen \"uber das {I}kosaeder und die {A}ufl\"osung der
  {G}leichungen vom f\"unften {G}rade}.
\newblock Leipzig: Teubner, 1884.

\bibitem{lieblich2007}
M.~Lieblich.
\newblock Moduli of twisted sheaves.
\newblock {\em Duke Math. J.}, 138(1):23--118, 2007.

\bibitem{schur68}
I.~Schur.
\newblock {\em Vorlesungen \"uber {I}nvariantentheorie}.
\newblock Bearbeitet und herausgegeben von Helmut Grunsky. Die Grundlehren der
  mathematischen Wissenschaften, Band 143. Springer-Verlag, Berlin, 1968.

\bibitem{shioda67}
T.~Shioda.
\newblock On the graded ring of invariants of binary octavics.
\newblock {\em Amer. J. Math.}, 89:1022--1046, 1967.

\bibitem{springer77}
T.~A. Springer.
\newblock {\em Invariant theory}.
\newblock Lecture Notes in Mathematics, Vol. 585. Springer-Verlag, Berlin,
  1977.

\end{thebibliography}
\end{document}